\documentclass{amsart}
\newtheorem{lemma}{Lemma}

\newtheorem{exam}{Example} 
\newtheorem{dfn}{Definition} 
\title{some ordered matrix semigroups}
\author {P. G. Romeo  and Sathi P H }


\address{Dept. of Mathematics, CUSAT, Kochi, Kerala INDIA.}

\email{romeopg@cusat.ac.in; sathiph123@gmail.com}
\begin{document}

	\begin{abstract}
		A semigroup together with compatible partial order is called an odered semigroup. In this paper we discuss the ordered matrix semigroups.
	\end{abstract}	

\maketitle
                       
\section{Intoduction and preliminaries}

A semigroup $(S,\cdot)$ is a non empty set together with an assosiativebinary operation $\cdot$. Obviously every group is a semigroup, however there are semigroups which fails to be a group such as integers under usual multiplicaton, natural numbers under adition and the like. More familiar examples of semigroups are 
\begin{enumerate}
	\item Full transformation semigroup $\mathcal{T}(X)$ of all transformations on a set $X$ with composition of transformations as binary operation.
	\item All square matrices over a field $F$ [ring $R$] $M_n(F)\,[M_n(R)]$ with the usual multiplication of matrices as binary operation.
\end{enumerate}
As a mathematical structure semigroups attracted much attention and has grown to an extend that it pocesses a vivd and deep structure theory.  Regular semigroups and inverse semigroups include the class of semigroups which are well developed. 
An element $a$ in a semigroup $S$ is said to be von Neumann regular if there exists an alement $x\in S$ such that $axa=a$. A semigroup $S$ is said to be regular if every element $a\in S$ are regular.
A regular semigroup in which the idempotent elements ( those elements $e$ in a semigroup $S$ such that $e\cdot e=e)$ commutes are called inverse semigroups. 
\begin{dfn} Let $S$ be a semigroup and $a\in S$.
	\begin{enumerate}
		\item A particular solution to $axa=a$ is called the inner inverse of $a$ and is denoted  by $a^-$.
		\item A solution of the equation $xax=x$ is called the outer inverse of $a$ and is denoted  by $a^=$
		\item An inner inverse of $a$  that is also an outer inverse is called a reflexive
		inverse and is denoted by $a^{+}$.
	\end{enumerate}
		The set of all inner (resp. outer, resp. reflexive) inverses of  $a$ is denoted by
		$a\{1\} (resp. a\{2\} , resp. a\{1,2\} )$
\end{dfn}

\begin{dfn}
	A semigroup $S$ is said to be weakly separative if, for any $a,b\in S$,
	$$ asa=asb=bsa=bsb \Rightarrow a=b.$$	
\end{dfn}

\begin{lemma}
Every regular semigroup (in the sense of von Neumann) is weakly separative.
\end{lemma}
\begin{proof}
Let $S$ be regular. For $a,b \in S,$ we have $axa=a$ and  $byb=b$ for some $x,y \in S$.
Assume that $asa=asb=bsa=bsb$ for all $s\in S$. In particular, $a=axa=ax$ and $ayb=byb=b$ so that $a=axb=ax(byb)=(axb)yb=ayb=b$.	
\end{proof}

\begin{exam}\item(1) The set of all transformations  on a set $X$ with composition of maps is a regular semigroup called the full transformation semigroup.
	\item (2) All $n\times n$ matrices $M_n$ over a field $F$ or a ring $R$ are regualr semigroups with respect to usual mutliplication of matrices.
\end{exam}

However, here it should be noted that the usual matrix product is not the only matrix product with respect to which $M_n$ is a semigroup. In fact there are other interseting matrix products such as 

\begin{enumerate}
	\item Kronecker product (Tensor product)
	\item Hadamard product (Schur product)
\end{enumerate}
with respect to which $M_n$ is a semigroup.

Consider $ A=[a_{ij}],B=[b_{ij}] \in M_n$, then their Kronecker product $A \otimes B$ is defined to be the  $n^{2}\times n^{2}$ matrix partitioned into $n^{2}$ blocks with the 
$(i,j)^{th}$ block as the $n\times n$ matrix $a_{ij}B$, i.e.,  

$$ \mathbf{{A}\otimes{B}} \;=\; 
\begin{pmatrix}
	a_{11}B &a_{12}B&\cdots&a_{1n}B \\
	a_{21}B &a_{22}B&\cdots&a_{2n}B \\
	\cdot&\cdot&\cdots &\cdot\\
	a_{n1}B &a_{n2}B&\cdots&a_{nn}B 
\end{pmatrix}. $$
Clearly $( A\otimes B)\otimes C=A \otimes (B\otimes C)$ for all $A,B,C \in M_n$, thus  $(M,\otimes)$ is a semigroup.\\ 

The Hadamard product or Schur product of $ A,B\in M_n$ is defined by the entrywise product $A\circ B=(a_{ij}b_{ij})$ that is, 

$$ \mathbf{{A}\circ{B}} \;=\; \begin{pmatrix}
	a_{11}b_{11} &a_{12}b_{12}&\cdots&a_{1n}b_{1n} \\
	a_{21}b_{21} &a_{22}b_{22}&\cdots&a_{2n}b_{2n} \\
	\cdot&\cdot&\cdots &\cdot\\
	a_{n1}b_{n1} &a_{n2}b_{n2}&\cdots&a_{nn}b_{nn} 
\end{pmatrix}  
$$

Since $A\circ (B\circ C)=(A \circ B)\circ C $ for all $A,B,C \in M_{n}(F)$,  we have $(M_n,\circ )$ is a semigroup with respect to the multiplication $\circ$.\\

A Partial order on a semigroup $S$ is a binary relation $\leq$ on $S$ which is reflexive, antisymmetric and transitive.
\begin{dfn}
An ordered semigroup $(S, \cdot, \leq)$ is a semigroup $(S,\cdot)$ together with a compatible partial order  $"\leq "$ on $S$ such that  for any $x,y,z \in S$, 
  $$ x \leq y \Rightarrow \, xz \leq xz\quad\text{and}\quad zx \leq zy.$$ 
\end{dfn}
  \begin{exam}
  	 The set of all relations on a set $X$ denoted by $B(X)$ is an ordered  semigroup with the composition of relations as the binary operation and inclusion as the compatible partial order on $B(X)$.  
  \end{exam}

Generally, any semigroup$(S,\cdot)$ can be considered as an ordered semigroup with respect to the order $\leq$ as the identity relation on $S$, that is  

$$x \leq y \Longleftrightarrow x=y.$$

 If $P$ is any partial order on a semigroup $S$ then the relation 
 $$P_{1}=\{(x,y)\,\vert\, (axb,ayb) \in  P\,\text{for all}\, (a,b) \in S\}$$ 
 is the largest compatible partial order contained in $P$. Thus, every semigroup can be endowed with a partial order so that $S$ becomes an ordered semigroup.

\begin{lemma}
Let $(S,\cdot,\leq)$ be an ordered semigroup. For every $s \in S$ define a binary composition $\rho$ on $S$ by
$$x\rho y = xsy \quad \text{for all}\quad  x,y \in S $$ 
Then, $(S, \rho )$ is a semigroup such that $(S, \rho,\leq)$ is an ordered semigroup and is called the $s$-invariant of $S$.
\end{lemma}

\begin{dfn}
	Let $(S, \cdot, \leq)$ be an ordered semigroup. A non empty subset $A\in S$ is a subsemigroup of $S$ if for any $a,b \in A$  implies $ab \in A$.
\end{dfn}

\section{Ordered Matrix Semigroups}

In the following we proceed to describe ordered semigroup $n\times n$ matrices. At the outset it is seen that the semigroup $M_{n}(\mathbb{R})$ square matrices over real nuumbers with order defined by, for $A,B\in M_{n}(\mathbb{R})$  with $A=[a_{ij}],B=[b_{ij}],\, 1\leq i\leq n,\, 1\leq j \leq n$, then 	
$$A \leq B \Rightarrow  a_{ij}\leq b_{ij}\, \text{for all}\quad i\quad \text{and}\, j $$
fails to be an ordered semigroup. However for $M_{n}(\mathbb{Z}_+)$ of all $n\times n$ matrices with non negative integer entries we have the following lemma.
\begin{lemma}
Consider the semigroup $M_{n}(\mathbb(\mathbb{Z}_+))$ of all $n\times n$ matrices with non negative integer entries. For $A,B\in M_{n}(\mathbb{Z}_+)$  with $A=[a_{ij}],B=[b_{ij}],\, 1\leq i\leq n,\, 1\leq j \leq n$, define the relation $\leq$ by 	
$$A \leq B \Rightarrow  a_{ij}\leq b_{ij}\, \text{for all}\quad i\quad \text{and}\, j $$
is a partial order with respect to which $M_{n}(\mathbb{Z}_+)$ is an ordered semigroup.
\end{lemma}

\begin{proof}
Clearly the order $\leq$ is reflexive as, $a_{ij} \leq a_{ij}$ for all $i$ and $j$.
Suppose $A\leq B$ and $B\leq A$, ie.,  $a_{ij} \leq b_{ij}$ and $b_{ij}\leq a_{ij}$ for all $i$ and $j$.This gives $a_{ij}=b_{ij}$ and hence anti symmetric. If $A \leq B$ and$ B\leq C$, then $a_{ij}\leq b_{ij}$ and $b_{ij}\leq c_{ij}$. This implies $a_{ij} \leq c_{ij}$,ie, $A \leq C$, hence transitive.Thus $\leq$ is a partial order.

To prove the compatibility, suppose $A\leq B$. ie,$a_{ij}\leq b_{ij}$. For $C=[c_{ij}]\in M$,
$$(CA)_{ij} =\sum c_{ik}a_{kj}\leq \sum c_{ik} b_{kj} =(CB)_{ij}\,\text{for}\,1\leq k \leq n$$
thus $CA\leq CB$. Similarly it is seen that $AC\leq BC$, ie., $\leq$ is compatible. 
\end{proof}. 

\subsection{Conrad order on Regular matrix semigroups}
 
 In  (cf. \cite{abi}) Abian described a partial order on a semiprime ring $R$  as follows 
 $$a\leq b \Longleftrightarrow arb=ara \quad \text{for all}\quad r\in R .$$ 
 In the light of this partial order, Conrad  defined a relation $\rho$ on a semigroup $S$ by 
 $$a\rho b\Leftrightarrow asa=asb=bsa \quad\text{for all}\quad s\in S $$  
 which turned out to be a partial order for  'weakly separative' semigroups.\\
 
Ler $R$ a ring. Consider matrix semigroup  $M_{n}(R)$. Now we proceed to describe the relation $\rho$ on $M_{n}(R)$ as follows, for $ A ,B \in M_{n}(R)$., 
$$A\rho B \Leftrightarrow ASA=ASB=BSA\quad  \forall \,\, S\in \,\,  M_{n}(R)$$
and this relation $\rho$ is called the Conrad relation on $M_{n}(R)$

\begin{lemma} The Conrad relation on $M_n (R)$ is a partial order with respect to which $M_n (R)$ is an ordered semigroup. \end{lemma}

\begin{proof} It is easy to see that the relation $\rho$ on $M_{n}(R)$ is reflexive. For transitivity, suppose $A \rho B$ and $B \rho C$, then $ASA=ASB=BSA$ and $BSB=BSC=CSB$ for all $S \in M_{n}(F)$. 
 For $P \in M_{n}(R),$ we obtain $(ASC)P(ASC)=(ASC)P(ASB)=(ASC)P(BSA)=(ASC)P(CSA) \Rightarrow ASC=CSA $.  Similarly, $(ASC)P(ASC)=(ASC)P(ASB)=(ASC)P(BSA)=(ASC)P(CSA) \Rightarrow ASC=CSA$. This gives $ASC=CSA=ASA$, ie, $A \rho C$.
	
 Burgess and Raphael proved that $\rho$ is a partial order on a semigroup $S$ if and only if it is weakly separative. Since $M_{n}(R)$ is a regular semigroup, obviously  $M_{n}(R)$ is weakly seperative and hence $\rho$ is a partial order.  
 Suppose $A \rho B$  and $C \rho D$. Then, $ASB=BSA=BSB$ and $CSC=CSD=DSC \quad \forall \, S \in M_{n}(F)$. 	Then,$(AC)S(BD)=(AC)S(AC)=(BD)S(AC) \Rightarrow AC \rho BD$, hence the compatibility and $(M_n (R),\rho)$ is an ordered matrix semigroup.
\end{proof}

\subsection{More  partial orders on $M_{n}(R)$}
 Let $R$ be any ring. Jerzy K Baksalary and Sujit Kumar Mitra introduced the following two orderings on $M_{n}(R)$

For $ A, B \in M_{n}(R)$, two orders defined by 
$$A\star \leq B \Leftrightarrow A^{\star}A=A^{\star}B\,\,\text{and}\,\, R(A)\subseteq R(B)$$ 
where $A^{*}$ stands for the conjugate transpose of $A$ and $R(A)$ is the range space of $A$. 
In a similar way another order is given by 
$$A \leq \star B\Leftrightarrow AA^{\star}= BA^{\star}\,\text{and}\, R(A^{*})\subseteq R(B^{\star}).$$
These orders are called the left star order and the right star order respectively.

\begin{lemma}
	The left star order and the right star order are partial orders on $M_{n}(R)$
\end{lemma} 
\begin{proof}
	Clearly $A \star\leq A $. Suppose $A \star \leq B$ and $B \star \leq A$. Then $A^{\star}A=A^{\star}B,\,\,R(A)\subseteq R(B),\,\,B^{\star}B=B^{\star}A$ and $R(B)\subseteq R(A)$. This gives $R(A)=R(B)$. Let  $A^{\dagger}$ be the Moore-Penrose  inverse of $A$ ( ie., the unique matrix satisfying $AA^{\dagger}A =A$, and $A^{\dagger}AA^{\dagger}=A^{\dagger}$ ). Then $AA^{\dagger}=(AA^{\dagger})^{\star}$ and $A^{\dagger}A=(A^{\dagger}A)^{\star}$ and 
	$$A=AA^{\dagger}A=(AA^{\dagger})^{\star}A=(A^{\dagger})^{\star}(A^{\star}A=(A^{\dagger})^{\star}A^{\star}B=(AA^{\dagger})^{\star}B=AA^{\dagger}B=B$$
	hence the relation $' \star \leq'$ is antisymmetric.
	
	For transitivity, let  $A \star \leq B$ and $B \star \leq C$. Then, $A^{\star}A=A^{\star}B,\, \,R(A) \subseteq R(B),  B^{\star}B = B^{\star}C$ and $R(B) \subseteq R(C)$, from this we get $R(A) \subseteq R(C)$. Now, $A^{\star}A=A^{\star}B=A^{\star}(BB^{\dagger}B)=A^{\star}(BB^{\dagger})^{\star}B=(A^{\star}(B^{\dagger})^{\star}B^{\star})B=A^{\star}(B^{\dagger})^{\star}(B^{\star}B)=A^{\star}(B^{\dagger})^{\star}(B^{\star}C)=(BB^{\dagger}A)^{\star}C=A^{\star}C$. Thus $A^{\star}A=A^{\star}C$ with $R(A) \subseteq R(C)$ implying that $A \star \leq C*$.
	
	 Similarly it can be seen that the right star order $"\leq \star"$ is a partial order on $M_{n}(R)$.
\end{proof}

\begin{lemma}
	$M_{n}(R)$ is an ordered semigroup under the left star ordering.
\end{lemma}
\begin{proof}
It is already seen that the left star order is a partial order on $M_{n}(R)$. Consider 
$A,B,C,D \in M_{n}(F)$ such that $A \star \leq B$ and $C \star \leq D$, then   $A^{\star}A=A^{\star}B,\,\,R(A) \subseteq R(B),\,\,C^{\star}C=C^{\star}D$ and $R(C) \subseteq R(D)$. \\
Further	
\begin{align*}
(AC)^{\star}AC=C^{\star}(A^{\star}A)C=C^{\star}(A^{\star}B)C &=(AC)^{\star}B(CC^{\dagger}C)=(AC)^{\star}B(CC^{\dagger})^{\star}C\\
&=(AC)^{\star}B(C^{\dagger})^{\star}(^{\star}C)\\
&=(AC)^{\star}B(C^{\dagger})^{\star}(C^{\star}D)\\
&=(AC)^{\star}B(CC^{\dagger})^{\star}=(AC)^{\star}B(CC^{\dagger}D)\\
&=(AC)^{\star}BD.		
\end{align*} 
and since $R(C) \subseteq R(D)$, we have $R(AC) \subseteq R(BD)$. Thus $AC\star\leq BD$ and hence  $(M_{n}(R),\cdot, \star \leq)$ is an ordered matrix semigroup.\\
Similarly we can prove that $(M_{n}(R),\cdot,\leq \star)$ is an ordered semigroup.
\end{proof}

\subsection{ Positive semidefinite matrices\,(PSD)}
From here onwards we restrict to the matrices $M_{n}(\mathbb{C})$. A matrix $A \in M_{n}(\mathbb{C})$ is said to be positive semi definite (positive definite) if $v^{\star}Av \geq 0 \,(v^{\star}Av > 0)$ for all $v \in \mathbb{C}_{n}$. We write $A \geq 0\,(A>0)$ to mean that $A$ is positive semi definite (positive definite).\\

In the following we list some properties of PSD matrice
\begin{enumerate}
	\item $A$ is PSD if and only if  it is Hermitian, ie., $A = A^{\star}$
	\item If $A$ is PSD, then all principal submatrices of $A$, $Trace\,(A),\,\mid A \mid$ and all principal minors of $A$ are PSD.
	\item $A$ is PSD if and only if $A=M^{\star}M$ for some matrix $M$.
	\item $A$ is PSD if and only if $A= P^{\star}P$ for some upper triangular matrix $P$ with non negative diagonal entries only(Cholesky decomposition of $A$
	\item $A$ and $B$ be congruent matrices then $A$ is PSD $\Leftrightarrow B$ is PSD.
	\item $A$ is PSD if and only if $A=B^{2}$ for some PSD matrix $B$. Then the unique $B=A^{1/2}$ is called the  positive square root of $A$.
	\item Schur product theorem: Let $A$ and $B$ be PSD matrices of size n. Then the Schur product  $A \circ B$ is also PSD.(but the conventional product neednot be PSD)
	\item Let $A$ be Hermitian and PSD. Then there exists a sequence of PSD matrices $A_{1},  A_{2}  ....$ such that $A_{k} \rightarrow A$ as $k \rightarrow \infty$. We can define $A_{k}=A+k^{-1}I$
\end{enumerate}

It is easy to observe that the set of all $n \times n$ Hermitian matrices denoted by $M_{n}^{H} \subset M_{n}(\mathbb{C})$. For $A,B \in M_{n}^{H},\, A \circ B$ is always Hermitian, ie., $M_{n}^{H}$ is a semigroup and is a subsemigroup of $M_{n}(\mathbb{C})$.

\subsection{ Loewner partial order on $M_{n}^{H}$}   
Every partial order $\preceq$ on a real linear space $S$ can be defined by $A \preceq B$ 
for all $A,B  \in S$ if their difference lies in a special closed convex cone. For the Loewner order the elements of the real linear space are the $n \times n$ Hermitian matrices and elements of the closed convex cone are the positive semi-definite matrices (PSD).\\

Also, the collection of all PSD matrices of order $n$ is a semigroup under '$\circ$' and is denoted by $M_{n}^{\geq}$. Then 
$$M_{n}^{\geq}= \lbrace K\in M_{n}\, \mid \,  K=LL^{\ast}\,\text{ for some}\, L \in M_{n}\rbrace. $$
For $A,B \in M_{n}^{\geq}$, $A \circ B \in  M_{n}^{\geq}$  and $A \circ (B \circ C)=(A \circ B) \circ C$ for all $A,B,C \in  M_{n} ^{\geq}$ and hence $ M_{n}^{\geq}$ is a semigroup. Now we define the Loewner order $\preceq$ on $M_{n}^{H}$ by $A \preceq B$ if and only if f $B-A$ is Hermitian and positive semidefinite.\\
\begin{lemma} The Loewner partial order 
	$\preceq$ is a compatible partial order on $M_{n}^{H}$.
\end{lemma}
\begin{proof}
	Clearly $0 \preceq A$ means that $A$ is  PSD.
	For $A \in M_{n}^{H}$,  $A-A=\textbf{0} \in M_{n}^{\geq}$ implying that $A \preceq A$.
	Suppose $A \preceq B$ and $B \preceq A$. Since  $A \preceq B \Rightarrow B-A \in M_{n}^{\geq}$, that is., $B-A=KK^{\ast}$ for some $K \in M_{n}$. 
	Similarly, $B \preceq A \Rightarrow A-B=LL^{\star}$ for some $L \in M_{n}$. Since, $A-B=-(B-A)$ we have  $LL^{\star}+ KK^{\star}=0 $ which implies 
	$LL^{\star}=KK^{\star}=0$, ie., $A-B=0=B-A \Rightarrow A=B$  proving the antisymmetry.
	
	For transitivity of $\preceq$, consider $A,B \in M_{n}^{\geq}$ such that $A \preceq B$ and $B \preceq C$, then, $B-A= KK^{\star}$ and $C-B= PP^{\star}$ for some $L,P \in  M_{n}$. Now, $C-A=(C-B)+(B-A) \in M_{n}^{\geq}$, being the sum of two PSD matrices. Thus $\preceq$ is a partial order.
	
	Let $A \preceq B$ and $C \preceq D$. Then $B-A= KK^{\star}$ and $D-C= PP^{\star}$ for some $L,P \in M_{n}$. Now, $B \circ D-A \circ C=B \circ (D-C)+(B-A) \circ C= B \circ (LL^{\star})+(KK^{\star}) \circ C \in M_{n}\geq$  as the Hadamard product and sum of two PSD matrices are PSD. This implies that $A \circ C , 
	 \preceq B \circ D$, that is., $\preceq$ is compatible under $\circ$.  Thus $(M_{n}^{H},\circ,\preceq)$ is an ordered matrix semigroup.
\end{proof} 

 Ordered semigroups have many applications in the theory of computer arithmetics, formal languages, error-correcting codes and the like. This class of semigroups is currently a hot topic of research and have been studied by several authors.

\end{document}